\pdfoutput=1
\documentclass[11pt]{amsart}
\usepackage{amsmath}
\usepackage[all]{xy}
\usepackage{amssymb}
\usepackage{mathrsfs}
\usepackage{cite}
\usepackage{hyperref}
\usepackage{amsfonts}
\usepackage{todonotes}
\usepackage{mathdots}
\theoremstyle{plain}
\newtheorem{thm}{Theorem}[section]
\newtheorem{prop}[thm]{Proposition}
\newtheorem{lem}[thm]{Lemma}
\newtheorem{cor}[thm]{Corollary}

\theoremstyle{remark}
\newtheorem{rem}[thm]{Remark}
\newtheorem*{acks}{Acknowledgements}

\theoremstyle{definition}
\newtheorem{defn}[thm]{Definition}

\newtheorem{eg}[thm]{Example}

\theoremstyle{conjecture}
\newtheorem{conj}[thm]{Conjecture}

\def\H{\mathsf{H}}
\def\Q{\mathbb{Q}}
\def\C{\mathbb{C}}
\def\Z{\mathbb{Z}}
\def\A{\mathcal{A}}
\def\B{\mathcal{B}}
\def\X{\mathcal{X}}
\def\U{\mathcal{U}}
\def\P{\mathcal{P}}
\def\O{\mathcal{O}}

\def\c{\mathcal{C}}
\def\b\X{\mathsf{D}^{\mathsf{b}}(\mathsf{X})}
\def\SOD{\mathsf{SOD}}
\def\1\b\Y{\mathsf{D}^{\mathsf{b}}(\mathsf{Y})}
\def\2\b\Z{\mathsf{D}^{\mathsf{b}}(\mathsf{Z})}
\def\CH{\mathsf{CH}}
\def\K{\mathsf{K}}
\def\HH{\mathsf{HH}}
\def\HN{\mathsf{HN}}
\def\Ch{\mathsf{Ch}}
\def\dgcat{\mathsf{dg-cat}}
\def\Mor{\mathsf{Mor}}
\def\Hom{\mathsf{Hom}}
\def\X{\mathsf{X}}
\def\E{\mathsf{E}}
\def\F{\mathsf{F}}
\def\G{\mathsf{G}}
\def\Td{\mathsf{Td}}
\def\v{\mathsf{v}}
\def\d{\mathcal{D}}
\def\x1{\mathsf{x}'}
\def\3\dg{\mathsf{dg}}
\def\Hmo{\mathsf{Hmo}}
\def\rep{\mathsf{rep}}
\def\PNChow{\mathsf{PChow}}
\def\f{\mathsf{f}}
\def\kC{\mathsf{C}(\mathsf{k})}
\def\4\dg{\mathsf{C}_{\mathsf{dg}}(\mathsf{k})}
\def\k{\mathsf{k}}
\def\j{\mathsf{j}}
\def\i{\mathsf{i}}
\def\Ho{\mathsf{Ho}}
\def\a{\mathsf{A}}
\def\Perf{\mathsf{Perf}}
\def\Per{\mathsf{Per}}
\def\mod{\mathsf{mod}}
\def\y{\mathsf{y}}
\def\Y{\mathsf{Y}}
\def\Chow{\mathsf{Chow}}
\def\NChow{\mathsf{NChow}}
\def\SmProj{\mathsf{SmProjec}}
\def\z{\mathsf{Z}}
\def\L{\mathsf{L}}
\def\R{\mathsf{R}}
\def\HC{\mathsf{HC}}
\def\per{\mathsf{per}}
\def\p{\mathsf{p}}
\def\Proj{\mathsf{Proj}}
\def\W{\mathsf{W}}
\def\n{\mathsf{n}}
\def\Pf{\mathsf{Pf}}
\def\V{\mathsf{V}}
\def\m{\mathsf{m}}
\def\r{\mathsf{r}}
\def\S{\mathsf{S}}
\def\q{\mathsf{q}}
\def\HPD{\mathsf{HPD}}
\def\Hodge{\mathsf{Hodge}}
\def\top{\mathsf{top}}
\newcommand \Caldararu {\text{C\u{a}ld\u{a}raru}}
\newcommand \Poincare {\text{Poincar\'{e}}}
\newcommand \Sym {\mathop\mathsf{Sym}}
\renewcommand\Im{\mathop\mathsf{Im}}

\numberwithin{equation}{section}

\begin{document}
\title{Noncommutative Hodge conjecture}

\author{Xun Lin}
\email{lin-x18@mails.tsinghua.edu.cn}
\address{Yau mathematical science center, Tsinghua university, Beijing China.}

\begin{abstract}
  The paper provides a version of the rational Hodge conjecture for $\3\dg$ categories. The noncommutative Hodge conjecture is equivalent to the version proposed in \cite{perry2020integral} for admissible subcategories. We obtain examples of evidence of the Hodge conjecture by techniques of noncommutative geometry. Finally, we show that the noncommutative Hodge conjecture for smooth proper connective $\3\dg$ algebras is true.
\end{abstract}

\maketitle

\tableofcontents

\section{Introduction}

Recently, G.\ Tabuada proposed a series of noncommutative counterparts of the celebrated conjectures, for example, Grothendieck standard conjecture of type $\mathsf{C}$ and type $\mathsf{D}$, Voevodsky nilpotence conjecture, Tate conjecture, Weil conjecture, and so on. After proposing the noncommutative counterparts, he proved additivity with respect to the
$\SOD$s (semi-orthogonal decomposition, see the notation Section \ref{notation}) for most of these conjectures. Then, he was able to give new evidence of the conjectures by a good knowledge of the semi-orthogonal decompositions of derived category of varieties. For the details, the reader can refer to ``Noncommutative counterparts of celebrated conjectures'' \cite{tabuada2019noncommutative}.

\par

In this paper, the author provides a version of the rational Hodge conjecture to the small $\3\dg$ categories. This new conjecture is equivalent to the classical Hodge conjecture when the $\3\dg$ category is $\Per_{\3\dg}(\X)$, where $\X$ is a projective smooth variety. It is equivalent to the version of Hodge conjecture in \cite{perry2020integral} for the admissible subcategories of $\b\X$.

\par

For $\operatorname{\Per}_{\3\dg}(\X)$, $\HH_{0}(\Per_{\3\dg}(\X))\cong \oplus \H^{\mathsf{p},\mathsf{p}}(\X,\C)$ by $\mathsf{HKR}$ isomorphism. In order to generalize the Hodge conjecture, we need to find natural intrinsic rational Hodge classes in $\HH_{0}(\A)$, and most importantly, it becomes the usual rational Hodge classes when $\A=\mathsf{Per}_{\3\dg}(\X)$. Classically, it is well known that the images of rational topological $\K$-groups under topological Chern character recovers the rational Betti cohomolgy. The topological $\K$-theory was generalized to the noncommutative spaces by A.\ Blanc\cite{blanc_2016}, it turns out that the image of rational topological $\K$-group $\K_{0}^{\mathsf{top}}(\A)_{\Q}$ under the topological Chern character becomes the even rational Betti cohomology when $\A=\mathsf{Per}_{\3\dg}(\X)$.

\par

There is a functorial commutative diagram$\colon$
$$\xymatrix{&&\HH_{0}(\A)\\
\K_{0}(\A)\ar[rru]^{\Ch}\ar[d]\ar[r]_{\Ch}&\HN_{0}(\A)\ar[d]^{j}\ar[ru]_{\pi}&\\
\K^{\top}_{0}(\A)\ar[r]^{\Ch^{\top}}&\HC^{\per}_{0}(\A)&}$$

\begin{defn}
Let $\A$ be a small $\3\dg$ category. The Hodge classes of $\A$ is defined as $$\Hodge(\A):=\pi(\j^{-1}(\Ch^{\top}(\K_{0}^{\top}(\A)_{\Q})))\subset \HH_{0}(\A).$$
\end{defn}
Clearly, the Chern character $\Ch: \K_{0}(\A)\rightarrow \HH_{0}(\A)$ maps $\K_{0}(\A)$ to $\Hodge(\A)$. We define the noncommutative Hodge conjecture for any $\3\dg$ categories as follow.

\begin{conj}
(Noncommutative Hodge conjecture)
The Chern character $\Ch: \K_{0}(\A)\mapsto \HH_{0}(\A)$ maps $\K_{0}(\A)_{\Q}$ surjectively into the Hodge classes $\Hodge(\A)$.
\end{conj}

 For the smooth proper $\3\dg$ categories, we propose an equivalent version of rational Hodge conjecture, for the reason that they are equivalent, see Remark \ref{remconj}. We write $\H$ as the isomorphism $\HC^{\per}_{0}(\A)\cong^{\H} \oplus \HH_{2\n}(\A)$ which is the Hodge decomposition by degeneration of noncommutative Hodge-to de Rham spectral sequence\cite{kaledin2016spectral}. Note that we choose a splitting. Define the rational classes in $\HC^{\per}_{0}(\A)$ as $\Ch^{\top}(\K^{\top}_{0}(\A)_{\Q})\cap \j (\HN_{0}(\A))$. Then we define the Hodge classes in $\HH_{0}(\A)$ as
 $$\Hodge(\A)=\mathsf{Pr}\circ\H(\Ch^{\top}(\K^{\top}_{0}(\A)_{\Q})\cap \j (\HN_{0}(\A))).$$
 Here the map $\mathsf{Pr}$ is the projection from $\oplus\HH_{2\n}(\A)$ to $\HH_{0}(\A)$. Clearly the natural Chern character map $\K_{0}(\A)_{\Q}$ to $\Hodge(\A)$.

\begin{defn}\label{smpconjA}(= Definition \ref{smpconj})
Hodge conjecture for smooth proper $\3\dg$ categories: the Chern character $\Ch: \K_{0}(\A)\rightarrow \HH_{0}(\A)$ maps $\K_{0}(\A)_{\Q}$ surjectively into the Hodge classes $\Hodge(\A)$.
\end{defn}
We prove that the noncommutative Hodge conjecture is equivalent to the classical Hodge conjecture when the $\3\dg$ category is $\Per_{\3\dg}(\X)$. The version of Hodge conjecture is equivalent with the one in \cite{perry2020integral} for admissible subcategories of $\b\X$, see Theorem \ref{admissible}.

\begin{thm}(=Theorem \ref{NHodge}).
   Let $\X$ be a smooth projective variety.
   $$\text{\it Hodge conjecture for}\ \X\ \Leftrightarrow\ \text{\it Noncommutative Hodge conjecture for}\ \Per_{\3\dg}(\X).$$
\end{thm}

 The author also proves that the Hodge conjecture is additive for geometric semi-orthogonal decomposition with independent method.

\begin{thm}(=Theorem \ref{SODHodge}).
 Suppose we have a nontrivial semi-orthogonal decomposition of derived category $\b\X=\langle \A,\B \rangle$ such that $\A$ and $\B$ are geometric, that is, $\B\cong \1\b\Y$ and $\A\cong \2\b\Z$ for some varieties $\Y$ and $\z$. Then, Hodge conjecture is true for $\X$ if and only if it is true for $\Y$ and $\z$.
\end{thm}

\begin{rem}
   We use this to obtain some results to prove that the commutative Hodge conjecture is a birational invariant for $4$ and $5$ dimensional varieties, see Theorem \ref{4fouldRinvariant}, which may be classically known for the experts, see also \cite{meng2019hodge}.
\end{rem}

After establishing the language of noncommutative Hodge conjecture, the author proves that the conjecture is additive for general $\SOD$s and the noncommutative motives.

\begin{thm}\label{sod}(=Theorem \ref{SODHodge}).
 Suppose we have a $\SOD$, $\b\X=\langle \A,\B \rangle$. There are natural $\3\dg$ liftings $\A_{\3\dg}$, $\B_{\3\dg}$ of $\A$, $\B$  corresponding to $\3\dg$ enhancement $\Per_{\3\dg}(\X)$ of $\b\X$.
   $$\text{\it Hodge conjecture for}\ \X\ \Leftrightarrow \text{\it Noncommutative Hodge conjecture for}\ \A_{\3\dg}\ \text{\it and}\ \B_{\3\dg}.$$
\end{thm}

 \begin{thm}\label{Nmotive}(=Theorem \ref{SODNMotive})
 Let $\A$, $\B$ and $\c$ be smooth and proper $\3\dg$ categories. Suppose there is a direct sum decomposition$\colon$
 $\U(\c)_{\Q}\cong \U(\A)_{\Q}\oplus \U(\B)_{\Q}$, see section \ref{section4.2} for the definition of $\U(\bullet)$ and $\U(\bullet)_{\Q}$. We have the following.
 $$\text{\it Noncommutative Hodge conjecture for}\ \c \Leftrightarrow  \text{\it Noncommutative Hodge conjecture for}\ \A\ and\ \B.$$
\end{thm}

Let $\A$ be a sheaf of Azumaya algebras on $\X$. Using work of G.\ Tabuada and Michel Van~den Bergh on Azumaya algebras\cite[Theorem 2.1]{tabuadavandenbergh2015},
$\U(\Per_{\3\dg}(\X,\A))_{\Q}\cong \U(\Per_{\3\dg}(\X))_{\Q}$. We have
the following.

\begin{thm}\label{Twisted}(=Theorem \ref{egTwistedscheme1})
  Noncommutative Hodge conjecture for $\Per_{\3\dg}(\X,\A)$ $\Leftrightarrow$ Noncommutative Hodge conjecture for $\Per_{\3\dg}(\X)$.
\end{thm}

This formulation of the noncommutative Hodge conjecture is compatible with the semi-orthogonal decompositions. Therefore, good knowledge of semi-orthogonal decomposition of varieties can simplify the Hodge conjecture, and gives new evidence of the Hodge conjecture. The survey ``Noncommutative counterparts of celebrated conjectures''\cite[Section 2]{tabuada2019noncommutative} provides many examples of the applications to the geometry for some conjectures via this approach. The examples also apply to the
noncommutative Hodge conjecture, and we give some further examples which are combined in the theorem below.

\begin{thm}\label{Example}
  Combining Theorem \ref{sod}, Theorem \ref{Nmotive}, and Theorem \ref{Twisted}, we have

 \begin{enumerate}
   \item \textbf{Fractional Calabi--Yau categories.}\\
   Let $\X$ be a hypersurface of degree $\leq \n+1$ in $\mathbb{P}^{\n}$. There is a semi-orthogonal decomposition
$$\Perf(\X)=\langle \mathcal{T}(\X),\O_{\X},\cdots,\O_{\X}(\n-\mathsf{deg}(\X))\rangle.$$
$\mathcal{T}(\X)$ is a fractional Calabi--Yau of dimension $\frac{(\n+1)(\mathsf{deg}-2)}{\mathsf{deg}(\X)}$\cite[Theorem 3.5]{Kuznetsov_2019}. We write $\mathcal{T}_{\3\dg}(\X)$ for the full $\3\dg$ subcategory of $\Per_{\3\dg}(\X)$ whose objects belong to $\mathcal{T}(\X)$. Then
$$\text{\it Hodge conjecture of}\ \X\Leftrightarrow \text{\it Noncommutative Hodge conjecture of}\ \mathcal{T}_{\3\dg}(\X).$$

   \item \textbf{Twisted scheme.}\\
   (A).\ Let $\X$ be a cubic fourfold containing a plane. There is a semi-orthogonal decomposition
 $$\Perf(\X)=\langle\Perf(\S,\A),\O_{\X},\O_{\X}(1),\O_{\X}(2)\rangle.$$
 $\mathsf{S}$ is a $\K_{3}$ surface, and $\A$ is a sheaf of Azumaya algebra over $\mathsf{S}$\cite[Theorem 4.3]{Kuznetsov2010DerivedCO}. Since the noncommutative Hodge conjecture is true for $\Per_{\3\dg}(\mathsf{S},\A)$ by Theorem \ref{Twisted}, hence the Hodge conjecture is true for $\X$.\\
   (B).\ Let $\f\colon \X\longrightarrow \mathsf{S}$ be a smooth quadratic
   fibration, for example, smooth quadric in relative projective space $\mathbb{P}^{\n+1}_{\mathsf{S}}$ \cite{Kuznetsov2005DerivedCO}. There is a semi-orthogonal decomposition
   $$\Perf(\X)=\langle \Perf(\mathsf{S},\mathsf{Cl}_{0}),\Perf(\mathsf{S}),\cdots,\Perf(\mathsf{S})\rangle.$$
   $\mathsf{Cl}_{0}$ is a sheaf of Azumaya algebra over $\mathsf{S}$ if the dimension $\n$ of the fiber of $\f$ is odd.

  \par

  Thus, if $\n$ is odd, the Hodge conjecture of $\X$ $\Leftrightarrow$ $\mathsf{S}$. Moreover, if $\dim \mathsf{S}\leq 3$, the Hodge conjecture for $\X$ is true.

   \item  \textbf{HP duality.}\\
   We write $\Hodge(\bullet)$ if the (noncommutative) Hodge conjecture is true for varieties (smooth and proper $\3\dg$ categories). Let $\Y\rightarrow\mathbb{P}(\V^{\ast})$ be the $\mathsf{HP}$ dual of $\X\rightarrow\mathbb{P}(\V)$, then $\Hodge(\X)\Leftrightarrow\Hodge(\Y)$. Choosing a linear subspace $\L\subset \V^{\ast}$. Let  $\X_{\L}=\X\times_{\mathbb{P}(\V)}\mathbb{P}(\L^{\perp})$ and $\Y_{\L}=\Y\times_{\mathbb{P}(\V^{\ast})}\mathbb{P}(\L)$ be the corresponding linear section. Assume $\X_{\L}$ and $\Y_{\L}$ are of expected dimension and smooth. If we assume $\Hodge(\X)$, then $\Hodge(\X_{\L})\Leftrightarrow \Hodge(\Y_{\L})$.
   \end{enumerate}
   \par
  \end{thm}
    We can prove (3) directly from the description of $\HPD$, see Theorem \ref{HPD}. For more examples constructed from $\HPD$, see Example \ref{egHPD}.
    Motivated from the noncommutative techniques, Theorem \ref{Example} (3), we expect that we can establish duality of
    the Hodge conjecture for certain linear section of the projective dual varieties by classical methods of algebraic geometry.

\begin{conj}(=Conjecture \ref{Conjprojectivedual})
 Let $\X\subset\mathbb{P}(\V)$ be a projective smooth variety. Suppose the Hodge conjecture is true for $\X$. Let $\Y\subset\mathbb{P}(\V^{\ast})$ be the projective dual of $\X\subset\mathbb{P}(\V)$. Choosing a linear subspace $\L\subset\V^{\ast}$.
  Suppose the linear section $\X_{\L}=\X\cap \mathbb{P}(\L^{\perp})$ and $\Y_{\L}=\Y\cap \mathbb{P}(\L)$ are both of expected dimension and smooth. Then, the Hodge conjecture of $\X_{\L}$ is equivalent to the Hodge conjecture of $\Y_{\L}$.
\end{conj}

Finally, we obtain some results by the algebraic techniques.  A $\3\dg$ algebra $\a$ is called connective if $\H^{\i}(\a)=0$ for $\i > 0$. According to \cite[Theorem 4.6]{raedschelders2020proper}, if $\a$ is a connective smooth proper $\3\dg$ algebra, then $\U(\a)_{\Q}\cong \U(\H^{0}(\a)/\mathsf{Jac}(\H^{0}(\a)))_{\Q}\cong \oplus \U(\C)_{\Q}$. Thus, we have the following.

\begin{thm}\label{algebra1}
 The noncommutative Hodge conjecture is true for smooth proper and connected $\3\dg$ algebra $\a$, see Theorem \ref{propersmoothconectivealgebraHodge}. In particular, the noncommutative Hodge conjecture is true for smooth and proper algebras.
\end{thm}
We also provide another proof for the case of smooth and proper algebras, see Theorem \ref{algebra}. Theorem \ref{algebra1} implies that if a variety $\X$ admits a tilting bundle (or sheaf), then the Hodge conjecture is true for $\X$, see the Corollary \ref{Tiltingsheaf} in the text.
\subsection*{Notation}\label{notation}
We assume the varieties to be defined over $\C$. We write $\SOD$ for semi-orthogonal decomposition of triangulated  categories. We say a semi-orthogonal decomposition is geometric if its components are equivalent to some derived categories of projective smooth varieties. We always assume the $\3\dg$ categories to be small categories.
We write $\k$ as the field $\C$ in some places without mentioning.

\begin{acks}
The author is grateful to his supervisor Will Donovan for helpful supports, discussions, and suggestions. The author would like to thank Anthony Blanc and Dmitry Kaledin for helpful discussions through E-mail. The author also thanks Shizhuo Zhang for informing the author about Alexander Perry's work  when the author finished most parts of the paper. The author is indebted to Alexander Perry for helpful comments and suggestions. The author thanks Michael Brown' comments, and pointing out a gap in the previous version concerning the issue of splitting of Hodge filtration. This leads the author to revising this new version which can avoid the issue of splitting of Hodge filtration.
\end{acks}

\vspace{5mm}

\section{Preliminary}
\subsection{The classical Hodge conjecture}
Given a projective smooth variety $\X$, there is a famous Hodge decomposition
$$\H^{\k}(\X(\C),\Z)\otimes\C \cong \oplus_{\p+\q=\k}\H^{\mathsf{p}}(\X,\Omega_{\X}^{\mathsf{q}})$$
where $\H^{\mathsf{p}}(\X,\Omega_{\X}^{\q})$ can be identified
with the $(\mathsf{p},\mathsf{q})$ classes in $\H^{\p+\q}(\X(\C),\C)$.
We define the rational (integral) Hodge classes as rational (integral) $(\mathsf{p},\mathsf{p})$ classes. By $\Poincare$ duality, there is a cycle map which relates the $\Chow$ group of $\X$ with its Betti cohomology
$$\mathsf{Cycle}\colon \quad \CH^{\ast}(\X)\longrightarrow \H^{\ast}(\X(\C),\C).$$
Clearly, the image lies in the integral Hodge classes. We obtain the rational cycle map when we tensor with $\Q$. The famous Hodge conjecture concerns whether the image of the (rational) cycle map is exactly the (rational) integral Hodge classes. It is well known that the integral Hodge conjecture is not true in general \cite{ATIYAH196225}, and the rational Hodge conjecture is still open. For more introductions to the classical Hodge conjecture, the reader can refer to the survey ``Some aspects of the Hodge conjecture'' \cite{voisin_2003}.

\begin{rem}
The rational (and integral) Hodge conjecture is true for weight one by Lefschetz one-one theorem. According to the $\Poincare$ duality, the rational Hodge conjecture is true for weight $\n-1$, $\n$ is the dimension of the variety. In particular, the rational Hodge conjecture is true for varieties of dimension less than or equal to 3.
\end{rem}

This paper focuses on the non-weighted rational Hodge conjecture. That is, we concern whether the rational cycle map maps $\CH^{\ast}(\X)_{\Q}$ surjectively into the rational Hodge classes.

\begin{thm}(Part of Grothendieck-Riemann-Roch [SGA6 exp.XIV]\cite{RR})\label{GRR}
  Let $\X$ be a smooth projective variety. There is a commutative diagram, where $\Ch_{\Q}$ are the certain Chern characters, $\K_{0}(\X)_{\Q}$ is the rational $0^{th}$ algebraic $\K$ group of the coherent sheaves.
  $$\xymatrix{\K_{0}(\X)_{\Q}\ar[r]^{\Ch_{\Q}}\ar[d]^{\cong}_{\Ch_{\Q}}&\H^{\ast}(\X,\C)\\
  \CH^{\ast}(\X)_{\Q}\ar[ru]_{\mathsf{cycle}}&}$$
\end{thm}

The image of the Chern character is in the rational Hodge classes, and the rational Hodge conjecture can be reformulated that $\mathsf{Ch}_{\Q}$ maps $\K_{0}(\X)_{\Q}$ surjectively into the rational Hodge classes.

\begin{prop}\label{corollary 2.1}
 We have the Mukai vector $\v(\bullet)$
 $$\v\colon \K_{0}(\X)\longrightarrow \oplus\H^{\mathsf{p},\mathsf{p}}(\X),\quad \quad \E\mapsto \mathsf{Ch}(\E)\sqrt{\Td(\X)}$$
 The non-weighted Hodge conjecture can be reformulated that $\v_{\Q}$ maps $\K_{0}(\X)_{\Q}$ surjectively into the rational Hodge classes.
\end{prop}

\begin{proof}
 There is a commutative diagram$\colon$
 $$\xymatrix@C6pc@R3pc{\K_{0}(\X)\ar[r]^{\v}\ar[rd]_{\Ch}&\oplus\H^{\p,\p}(\X)\ar[d]_{\cong}^{\frac{1}{\sqrt{\mathsf{Td}(\X)}}}\\
 &\oplus\H^{\p,\p}(\X)}$$
 Since the vertical morphism is an isomorphism which preserves the rational Hodge classes, $\v_{\Q}~$ maps $\K_{0}(\X)_{\Q}$ surjectively into the rational Hodge classes if and only if $\Ch_{\Q}$ maps $\K_{0}(\X)_{\Q}$ surjectively into the rational Hodge classes. Thus, the statement follows from the Theorem \ref{GRR} above.
\end{proof}

\subsection{Noncommutative geometry}
 We briefly recall the theory of noncommutative spaces.
 We regard certain $\3\dg$ categories as noncommutative counterparts of varieties. We will recall the basic notions. For survey of the $\3\dg$ categories, the reader can refer to the survey by B.\ Keller, ``On differential graded categories''  \cite{Keller2006OnDG}.

\begin{defn}
 The $\C$-linear category $\mathcal{A}$ is called a $\3\dg$ category if
 $\Mor(\bullet,\bullet)$ are differential $\mathbb{Z}$-graded $\k$-vector spaces. For every objects $\E$, $\F$, $\G$ $\in$ $\mathcal{A}$, the compositions
 $$\Mor(\F,\E)\otimes \Mor(\G,\F)\rightarrow \Mor(\G,\E) $$
 of complexes are associative.
 Furthermore, there is a unit $\k \rightarrow \Mor(\E,\E)$. Note
that the composition law implies that $\Mor(\E,\E)$ is a differential graded algebra.
\end{defn}

\begin{eg}
A basic example of $\3\dg$ categories is $\4\dg$, whose objects are complexes of $\k$-~vector space. The morphism spaces
are refined as follows$\colon$
\par
Let $\E,\F\in \mathsf{C}_{\3\dg}(\k)$, define degree $\n$ piece of the morphism $\Mor(\E,\F)$ to be $\Mor(\E,\F)(\n):=\Pi \Hom(\E_{\i},\F_{\i+\n})$.
The $\n^{\text{th}}$ differential is given by $\mathsf{d}_{\n}(\f)= \mathsf{d}_{\E}\circ \f - (-1)^{\n}\f\circ \mathsf{d}_{\F}$, $\f\in \Mor(\E,\F)(\n)$.
\end{eg}

\begin{defn}
 We call $\F\colon \c\longrightarrow \d$ a dg functor between $\3\dg$ categories if $\F\colon \Hom(\E,\G)\longrightarrow \Hom(\F(\E),\F(\G))$ is in $\kC$ (morphisms are morphism of chain complexes), $\E$, $\G \in \c$. We call $\F$ to be quasi-equivalent if $\F$ induces isomorphisms on homologies of morphisms and equivalence on their homotopic categories.
\end{defn}

\begin{defn}
The $\3\dg$ functor $\F\colon\A\longrightarrow \B$ is derived Morita equivalent if it induces an equivalence of derived categories by composition
$$\F^{\ast}\colon\mathsf{D}(\B)\cong \mathsf{D}(\A).$$
Note that if $\3\dg$ functor $\A \longrightarrow \B$ is a quasi-equivalence, then it is derived Morita equivalent, the reader can refer to ``Categorical resolutions of irrational singularities''\cite[Proposition 3.9]{Kuznetsov2015CategoricalRO} for an explicit proof.
\end{defn}

We consider the category of small $\3\dg$ categories, whose morphisms are the $\3\dg$ functors. It is written as $\dgcat$. According to G.\ Tabuada \cite{10.1155/IMRN.2005.3309}, there is a model structure on $\dgcat$ with derived Morita equivalent $\3\dg$ functors as weak equivalences. We write $\Hmo(\dgcat)$ as the associated homotopy category for such model structure. Given two $\3\dg$ categories $\A$ and $\B$, we have a bijection $\Hom_{\Hmo}(\A,\B)\cong \mathsf{Iso}\ \rep(\A^{op}\otimes^{\L} \B)$, where $\rep(\A^{op}\otimes^{\L} \B)$ is the subcategory of $\mathsf{D}(\A\otimes^{\L}\B)$ with bi-module $\X$ such that $\X(\A,\bullet)$ is a perfect $\B$ module. Linearizing the category, we obtain $\Hmo_{0}$ whose morphism spaces become $\K_{0}(\rep(\A^{op}\otimes \B))$. After $\Q$ linearization and idempotent completion, we get the category of  pre-noncommutative motive $\PNChow_{\Q}$.

\begin{defn}\label{additiveinvariant}
Any functor to an additive category $\c$, $\F\colon \dgcat \longrightarrow \c$, is called an additive invariant in the sense of G.\ Tabuada \cite{10.1155/IMRN.2005.3309} if $\colon$\\
(1) It maps the Morita equivalences to isomorphisms.\\
(2) For pre-triangulated $\3\dg$ categories $\A$, $\B$ and $\X$ with natural morphism $\i\colon\A\longrightarrow \X$ and $\j\colon\B\longrightarrow \X$ which induces semi-orthogonal decomposition of triangulated categories $\Ho(\X)=\langle \Ho(\A),\Ho(\B)\rangle$, there is an isomorphism $\F(\X)\cong \F(\A)\oplus \F(\B)$ which is induced by $\F(\i)+\F(\j)$.
\end{defn}

The following theorem is due to G.\ Tabuada.

\begin{thm}(G.\ Tabuada\cite[Theorem 4.1]{10.1155/IMRN.2005.3309}\label{theorem 1.21})
 The functor $\F$ in Definition \ref{additiveinvariant} that induces $\Hmo\longrightarrow \A$ is an additive invariant if and only if it factors through $\Hmo\longrightarrow \Hmo_{0}\longrightarrow \A$. That is, $\Hmo_{0}$ plays a role as the usual motives, and the additive invariants should be regarded as noncommutative Weil cohomology theories.
\end{thm}

\begin{rem}
  Due to many people's works, see a survey \cite{NM}, the Hochschild homology, algebraic $\K$-theory, (periodic) cyclic homology theory are all additive invariants. The Hochschild homology of proper smooth variety is the noncommutative counterpart of Hodge cohomology, and periodic cyclic homology corresponds to the de Rham cohomology.
\end{rem}

Given a proper smooth variety $\X$, there is a natural $\3\dg$ enhancement $\Per_{\3\dg}(\X)$, which is a $\3\dg$ enhancement of $\Perf(\X)$. In this sense, the $\3\dg$ categories can be regarded as noncommutative counterpart of varieties. In order to focus on the nice spaces, for example, the $\Chow$ motive concerns the proper smooth varieties, we restrict the $\dgcat$ to the smooth proper $\3\dg$ categories.

\begin{defn}
 $\mathsf{Dg}$ category $\A$ is called smooth if $\A$ is perfect $\A-\A$ bi-module. It is called smooth and proper if $\A$ is derived Morita equivalent to a smooth $\3\dg$ algebra of finite type.
\end{defn}

It is well known that the property of $\3\dg$ categories being smooth and proper is closed under derived Morita equivalence and tensor product \cite[Chapter 1, Theorem 1.43]{NM}.  People also define the properness as $\Hom_{\A}(\bullet,\bullet)$ being perfect $\k-\mod$. According to a book of G.\ Tabuada, ``Noncommutative motive'' \cite[Proposition 1.45]{NM}, such a definition of smooth and properness is equivalent to our definition.

\begin{defn}[Noncommutative $\Chow$ motive]\label{NMotives}
\cite{NM} We write $\Hmo^{\mathsf{sp}}_{0}$ as a full sub-category of $\Hmo_{0}$ whose objects are smooth proper $\3\dg$ categories. $\Q$ linearizing the category
$\Hmo^{\mathsf{sp}}_{0}$, that is, the morphisms become $\K_{0}(\A^{\mathsf{op}}\otimes \B)_{\Q}$\cite[Cor 1.44]{NM}, we obtain $\Hmo^{\mathsf{sp}}_{0,\Q}$. Then, we define $\NChow_{\Q}$ to be idempotent completion of $\Hmo^{\mathsf{sp}}_{0,\Q}$.
\end{defn}

There is a universal additive invariant$\colon$
$$\U\colon \operatorname{\dgcat}^{\mathsf{sp}}\longrightarrow \operatorname{\NChow}.$$
Let $\underline{\C}$ be the category with one object whose morphism space is $\C$. Then for any $\A\in \operatorname{\dgcat}$, $\Hom_{\NChow} (\U(\C),\U(\A))\cong \K_{0}(\rep(\A))\cong \K_{0}(\A):=\K_{0}(\mathsf{D}^{c}(\A))$. Since we have a functorial morphism $\Hom_{\NChow} (\U(\C),\U(\A))\longrightarrow \Hom_{\C}(\HH_{0}(\C),\HH_{0}(\A))$, there is a Chern character map $$\Ch\colon \K_{0}(\A)\longrightarrow \HH_{0}(\A).$$
Given any $\A$ module $X\in \mathsf{D}^{c}(\A)$ , it is defined via the following diagram of $\3\dg$ categories.
$$\xymatrix{&\Per_{\3\dg}(\A)\\
\underline{\C}\ar[ru]^{X}&\A\ar[u]}$$
It induces morphisms of Hochschild complexes naturally, and then an element in $\HH_{0}(\A)$ via isomorphism $\HH_{0}(\A)\cong \HH_{0}(\Per_{\3\dg}(\A))$. The isomorphism is because the Yoneda embedding $\A\longrightarrow \Per_{\3\dg}(\A)$ is a derived Morita equivalence. $\Per_{\3\dg}(\A)$ is defined as a full subcategory of $\3\dg$ $\A$ module whose objects are isomorphic to objects in $\Perf(\A)$.

\par

In general, given any additive invariant $\F$ with $\F(\k)\cong \k$, we have a Chern character map $\K_{0}(\A)\longrightarrow \F(\A)$. For example, the (periodic) cyclic homology, and the negatiave cyclic homology.
\par
It is natural to ask what are the relations between $\Chow$ motive $\Chow_{\Q}$ and noncommutative $\Chow$ motive $\NChow_{\Q}$. There is a nice answer due to remarkable works of Kontsevich and G.~Tabuada.

\begin{thm}\label{ChowNChow} (\cite[Theorem 1.1]{tabuada2011chow})
 There is a symmetric monoidal functor
 $$\phi \colon \SmProj^{\mathsf{op}}\longrightarrow \dgcat^{\mathsf{op}},\ \X\mapsto \Per_{\3\dg}(\X)$$
 such that the natural diagram is commutative.
 $$\xymatrix{\SmProj^{\mathsf{op}}\ar[r]^{\phi}\ar[d]&\dgcat^{\mathsf{sp}}\ar[d]\\
             \operatorname{\Chow}_{\Q}\ar[d]&\Hmo^{\mathsf{sp}}_{0}\ar[d]\\
             \operatorname{\Chow}_{\Q}/-\otimes\Q(1)\ar[r]^{\phi'}& \operatorname{\NChow}_{\Q}\subset \Hmo^{\ast}_{0,\Q}}$$
\end{thm}

With this commutative diagram, G.\ Tabuada was able to generalize some
famous conjectures to the noncommutative spaces, see ``Noncommutative counterparts of celebrated conjectures''~\cite{tabuada2019noncommutative}.

\vspace{5mm}

\section{Hodge conjecture and geometric semi-orthogonal decompositions }
In this section, we prove that the Hodge conjecture is additive for the geometric semi-orthogonal decompositions. In particular, the Hodge conjecture is a derived invariant.

\par

\begin{thm}\label{GSODHodge1}
 Suppose we have a nontrivial semi-orthogonal decomposition of derived categories $\b\X=\langle \A,\B \rangle$ such that $\A$ and $\B$ are geometric, that is, $\B\cong \1\b\Y$ and $\A\cong \2\b\Z$ for some varieties $\Y$ and $\z$. Then Hodge conjecture is true for $\X$ if and only if it is true for $\Y$ and $\z$.
\end{thm}

\begin{proof}[Proof]
Let's assume $\j\colon \2\b\Z\hookrightarrow\b\X$ to be an embedding with left adjoint $\L$, $\i\colon \1\b\Y\hookrightarrow \b\X$ with right adjoint $\R$. According to D.\ Orlov \cite[Theorem 2.2]{1996alg.geom..6006O}, they are all Fourier-Mukai functors. There is a diagram of triangulated categories$\colon$
$$\xymatrix{\1\b\Y \ar[r]_{\i} &\b\X \ar[r]_{\L}\ar@/_/[l]_{\R}&\2\b\Z\ar@/_/[l]_{\j}
}$$
with $\R\circ \i\cong \mathsf{id}$, $\L\circ \j\cong \mathsf{id}$, $\R\circ \j\cong 0$ and $\L\circ \i\cong 0$. Apply $0^{\mathsf{th}}$ $\K$-theory and $0^{\mathsf{th}}$ Hochschild homology theory, there are diagrams
$$\xymatrix{\K_{0}(\1\b\Y)\ar[r]_{\i} &\K_{0}(\b\X)\ar[r]_{\L}\ar@/_/[l]_{\R}&\K_{0}(\2\b\Z)\ar@/_/[l]_{\j}
}.$$
$$\xymatrix{\HH_{0}(\Y)\ar[r]_{\i_{\H}} &\HH_{0}(\X)\ar[r]_{\L_{\H}}\ar@/_/[l]_{\R_{\H}}&\HH_{0}(\z)\ar@/_/[l]_{\j_{\H}}
}.$$
Here we define $\HH_{0}(\bullet)$ as a subspace of de Rham cohomology. For example, $\HH_{0}(\X):=\oplus_{\p}\H^{\p,\p}(\X)\hookrightarrow \H^{\mathsf{even}}_{\mathsf{DR}}(\X)$. The morphisms of $\HH_{0}$ are induced by the Mukai vector of the corresponding kernel of functors. For example, take $\E\in \mathsf{D}^{\mathsf{b}}(\X\times\Y)$, then
$$\Phi_{\v(\E)}: \HH_{0}(\X)\longrightarrow \HH_{0}(\Y)$$
is defined as $\q_{\ast}(\p^{\ast}(\bullet)\cup \v(\E))$. Firstly, $\Phi_{\v(\E)}$ should induce morphism of de Rham cohomology, it is easy to prove that
$\Phi_{\v(\E)}$ maps $\HH_{\ast}(\X)=\oplus_{\p-\q=\ast}\H^{\p,\q}(\X)$ to $\HH_{\ast}(\Y)=\oplus_{\p-\q=\ast}\H^{\p,\q}(\Y)$. The reader can also see proof in \cite[Proposition 5.39]{bookHuybrechts}.

$$\xymatrix{&\X\times\Y\ar[dl]_{\p}\ar[dr]^{\q}&\\
\X&&\Y}$$
The morphisms of $\K_{0}$ groups are induced by the Fourier-Mukai functor.
According to \cite[Chapter 5, Section 5.2]{bookHuybrechts}, the Mukai vector $\v$ is compatible with morphism of $\K_{0}$-theory, namely, we have a diagram
  $$\xymatrix{\K_{0}(\Y)\ar[r]_{\i}\ar[d]_{\v_{\Y}}&\K_{0}(\X)\ar[r]_{\L}\ar@/_/[l]_{\R}\ar[d]_{\v_{\X}}&\K_{0}(\z)\ar[d]_{\v_{\z}}\ar@/_/[l]_{\j}\\
\HH_{0}(\Y)\ar[r]_{\i_{\H}}&\HH_{0}(\X)\ar[r]_{\L_{\H}}\ar@/_/[l]_{\R_{\H}}&\HH_{0}(\z)\ar@/_/[l]_{\j_{\H}}}$$
The morphisms $\R_{\H}$, $\i_{\H}$, $\j_{\H}$, and $\L_{\H}$  preserve rational classes. We first prove that $\i_{\H}+\j_{\H}$ induces an isomorphism of Hochschild homologies. Clearly $\i+\j$ is an isomorphism of $\K_{0}(\X)$ groups. Since Hochschild homology is an additive invariant, we have a non-canonical isomorphism $\HH_{0}(\X)\cong \HH_{0}(\Y)\oplus \HH_{0}(\z)$, which implies $\dim_{\C}\HH_{0}(\X)=\dim_{\C}\HH_{0}(\Y)+\dim_{\C}\HH_{0}(\z)$. This was proved by classical $\3\dg$ methods and the $\mathsf{HKR}$ isomorphism. The reader can also refer to A.\ Kuznetsov's paper ``Hochschild homology and semi-orthogonal decomposition'' \cite[Theorem 7.3(i)]{2009arXiv0904.4330K}.
\par
Since $\i_{\H}$ and $\j_{\H}$ are injective, which will be proved below, therefore $\i_{\H}+\j_{\H}$ being an isomorphism is equivalent to the fact that $\mathsf{Im}(\i_{\H})\cap \mathsf{Im}(\j_{\H})=0$. It suffices to prove that $\L_{\H}\circ \i_{\H}= 0$. If this is true, let $\alpha \in \mathsf{Im}(\i_{\H})\cap \mathsf{Im}(\j_{\H})$, then $\alpha= \i_{\H}\alpha_{\Y}=\j_{\H}\alpha_{\z}$, therefore $\L_{\H}\alpha = (\L_{\H}\circ \i_{\H})\alpha_{\Y}=(\L_{\H}\circ \j_{\H})\alpha_{\z}=\alpha_{\z}=0$, hence $\alpha=0$. In order to prove the claim $\L_{\H}\circ \i_{\H}=0$, we need the following lemma.

\begin{lem}\label{lem 2.2}
 Suppose an object $\E\in \mathsf{D}^{\mathsf{b}}(\mathsf{X}\times \mathsf{Y})$ induces a trivial Fourier--Mukai transform $\Phi_{\E} \colon \b\X\longrightarrow \1\b\Y$, then $\E\cong 0 \in \mathsf{D}^{\mathsf{b}}(\mathsf{X}\times \mathsf{Y})$.
\end{lem}

\begin{proof}[Proof of the lemma]
  Given any closed point $\mathsf{x}\in \X$, we have a natural closed embedding $\mathsf{l}_{\mathsf{x}}\colon \mathsf{x}\times \Y\hookrightarrow \X\times \Y$, and a simple calculation shows that $\Phi_{\E}(\k(\mathsf{x}))\cong \mathbb{L}\mathsf{l}_{\mathsf{x}}^{\ast}\E$ via identifying $\mathsf{x}\times \Y$ with $\Y$. Therefore, $\Phi_{\E}$ being trivial implies that $\mathbb{L}\mathsf{l}_{\mathsf{x}}^{\ast}\E$ is trivial. Since this is true for any closed points of $\X$, support of $\E$ is empty, which implies $\E\cong 0$.
\end{proof}

Back to the proof of Theorem \ref{GSODHodge1}. Since the functor $\L\circ \i\cong 0$ as Fourier--Mukai functor, by lemma above the kernel corresponding to $\L\circ \i$ is trivial. In particular, its Mukai vector is trivial, hence $\L_{\H}\circ\i_{\H}=0$.
\par
Now it is prepared enough to prove Theorem \ref{GSODHodge1}. Suppose Hodge conjecture for $\X$. Let $\alpha_{\Y}\in \oplus\H^{\mathsf{p},\mathsf{p}}(\Y,\Q)$, consider $\alpha=\i_{\H}\alpha_{\Y}\in \oplus\H^{\mathsf{p},\mathsf{p}}(\X,\Q)$. Since Hodge conjecture holds for $\X$, there exists an $\E\in \K_{0}(\X)_{\Q}$ such that $\v(\E)= \alpha$. Let $\E_{\Y}= \R(\E)$, then the image of $\v(\E_{\Y})$ and $\alpha_{\Y}$ under $\i_{\H}$ coincide. Since $\R_{\H}\circ \i_{\H}= \mathsf{id}_{\H}$, then $\i_{\H}$ is an injective morphism, therefore $\v(\E_{\Y})=\alpha_{\Y}$. This implies Hodge conjecture for $\Y$. The Hodge conjecture is true for $\z$ by the similar argument.

\par

Suppose Hodge conjecture is true for $\Y$ and $\z$, we prove that it is also true for $\X$. Let $\alpha \in \oplus \H^{\mathsf{p},\mathsf{p}}(\X,\Q)$, consider $\R_{\H}(\alpha)\in \HH_{0}(\Y)_{\Q}$ and $\L_{\H}(\alpha)\in \HH_{0}(\z)_{\Q}$. Since the Hodge conjecture is true for $\Y$ and $\z$, there exists an $\E_{\Y}\in \K_{0}(\Y)_{\Q}$ and an $\E_{\z}\in \K_{0}(\z)_{\Q}$ such that $\v(\E_{\Y})= \R_{\H}(\alpha)$, $\v(\E_{\z})=\L_{\H}(\alpha)$. Define $\alpha'=\i_{\H}\circ \R_{\H}(\alpha)+ \j_{\H}\circ \L_{\H}(\alpha)$. We prove that $\alpha'=\alpha$. Since $\i_{\H}\oplus \j_{\H}$ induces an isomorphism, there exist $\alpha_{1}\in \HH_{0}(\Y)$ and $\alpha_{2}\in \HH_{0}(\z)$ such that $\alpha= \i_{\H}(\alpha_{1})+ \j_{\H}(\alpha_{2})$. Applying morphism $\R_{\H}$, we obtain $\alpha_{1}=\R_{\H}(\alpha)$. Apply morphism  $\L_{\H}$, we obtain $\alpha_{2}=\L_{\H}(\alpha)$. Thus $\alpha = \i_{\H}\circ \R_{\H}(\alpha)+\j_{\H}\circ \L_{\H}(\alpha)$. Define $\E= \i(\E_{\Y})+\j(\E_{\z})\in \K_{0}(\X)_{\Q}$, then $\v(\E)= \v(\i(\E_{\Y}))+\v(\j(\E_{\z}))=\i_{\H}(\R_{\H}(\alpha))+\j_{\H}(\L_{\H}(\alpha))= \alpha$.
\end{proof}

\begin{rem}
  The statement of the theorem is still true if there is a semi-orthogonal decomposition of $\b\X$ that has more than two components. The proof is essentially the same.
\end{rem}

\begin{cor}\label{derivedinvariant}
 If $\b\X\cong \1\b\Y$, then Hodge conjecture of $\X$ $\Leftrightarrow$ Hodge conjecture of ~$\Y$.
\end{cor}

\begin{cor} \label{EC}
Suppose $\b\X$ admits a full exceptional collection, then the Hodge conjecture is true for $\X$.
\end{cor}

\begin{eg}
  The Grassmannians \cite{Kapranov_1985}, certain homogeneous spaces (see a brief survey in \cite[Section 1.1]{Kuznetsov_2016}), and smooth projective toric varieties\cite[Theorem 1.1]{2005math......3102K} admit full exceptional collection, hence the Hodge conjecture is true for these examples.
\end{eg}

\begin{eg}
  Let $\mathsf{M}=\mathbb{P}^{1}\times\cdots_{\n}\times \mathbb{P}^{1}$ with polarization $\mathcal{L}=\mathcal{O}(\mu_{1})^{\boxtimes}$ for a sequence of positive number
  $\mu=(\mu_{1},\cdots,\mu_{\n})$. Consider the obvious equivariant structure of group $\mathsf{PGL}_{2}$. Then the Mumford GIT quotient $\X(\mu)=\mathsf{M}//_{\mathcal{L}}\mathsf{PGL}_{2}$ admits a full exceptional collection for generic $\mu$ \cite[Section 6]{BallardFaveroKatzarkov+2019+235+303}. It is interesting that for finitely many $\mu$, $\X(\mu)$ is isomorphic to a Ball quotient by a classical result of
  Deligne and Mostow \cite{PMIHES_1986__63__5_0}.
\end{eg}

\begin{rem}
 If $\langle \E_{1},\E_{2},\cdots, \E_{\m} \rangle$ is a full exceptional collection of $\b\X$, then according to the proof in Theorem \ref{GSODHodge1},
 $\{\Ch(\E_{\i})\}_{\i=1}^{\m}$ forms a basis of $\mathsf{Hodge}(\X,\Q)$.
\end{rem}

\begin{eg}
 Let $\X$ be the projective space $\mathbb{P}^{n}$. There is a semi-orthogonal decomposition $\b\X=\langle \O,\O(1),\cdots, \O(n)\rangle$. We assume $n=3$ for simplicity. Since $\O(\i)$ is a line bundle, $c_{\j}(\O(\i))=0$, $\j\geq 2$. Write $\mathsf{H}$ as hyperplane of $\mathbb{P}^{3}$, then $\Ch(\O(\i))=1+\i\cdot \mathsf{H}+\frac{\i^{2}}{2}\cdot \mathsf{H}^{2}+\frac{\i^{3}}{6}\cdot \mathsf{H}^{3}$, $\HH_{0}(\mathbb{P}^{3})_{\Q}\cong\Q\oplus \Q \mathsf{H}\oplus \Q \mathsf{H}^{2}\oplus \Q \mathsf{H}^{3}$. The vectors $\Ch(\O)$, $\Ch(\O(1))$, $\Ch(\O(2))$, and $\Ch(\O(3))$ are linear independent which generate $\HH_{0}(\mathbb{P}^{3})_{\Q}$.
\end{eg}

\vspace{5mm}

\section{Noncommutative Hodge conjecture}
In this section, we propose the noncommutative Hodge conjecture, and prove that the noncommutative Hodge conjecture is additive for semi-orthogonal decomposition. We obtain more evidence of the Hodge conjecture via good knowledge of semi-orthogonal decomposition. Finally, we prove that the noncommutative Hodge conjecture is true for smooth proper connective $\3\dg$ algebras.

\par

\subsection{Formulation}

\begin{defn}
Let $\A$ be a small dg category. The Hodge classes of $\A$ is defined as $$\Hodge(\A):=\pi(\j^{-1}(\Ch^{\top}(\K_{0}^{\top}(\A)_{\Q})))\subset \HH_{0}(\A).$$
$$\xymatrix{&&\HH_{0}(\A)\\
\K_{0}(\A)\ar[rru]^{\Ch}\ar[d]\ar[r]_{\Ch}&\HN_{0}(\A)\ar[d]^{j}\ar[ru]_{\pi}&\\
\K^{\top}_{0}(\A)\ar[r]^{\Ch^{\top}}&\HC^{\per}_{0}(\A)&}$$
\end{defn}

\begin{conj}\label{Mainconj}
(Noncommutative Hodge conjecture)
The Chern character $\Ch: \K_{0}(\A)\mapsto \HH_{0}(\A)$ maps $K_{0}(\A)_{\Q}$ surjectively into the Hodge classes $\Hodge(\A)$.
\end{conj}

\begin{rem}
 Note that we obtain the abstract rational Hodge classes in $\HH_{0}(\A)$. Classically, the Hodge conjecture concerns the weight. However, to the author's knowledge, we don't know how to
 obtain the weight of the abstract Hodge classes. In the paper, we always assume the conjecture as a non-weighted Hodge conjecture.
\end{rem}

\begin{thm}\label{admissible}
The Conjecture \ref{Mainconj} is equivalent to the one in A.\ Perry's paper \cite[Conjecture 5.11]{perry2020integral} in the case of admissible subcategories of $\b\X$.
\end{thm}

\begin{proof}
 For the admissible subcategories of $\b\X$, the Hodge classes are defined as the classes of $\Ch^{\top}(\K_{0}^{\top}(\A)_{\Q})$ in $\HC^{\per}_{0}(\A)$ that lie in $\HH_{0}(\A)$ under the Hodge decomposition\cite{kaledin2016spectral}. The map $\j: \HN_{0}(\A)\rightarrow \HC^{\per}_{0}(\A)$ is injective by degeneration of noncommutative Hodge-to de Rham spectral sequence. Choose a splitting of the Hodge decomposition of $\HC^{\per}_{0}(\A)$ (the one in \cite{perry2020integral}), and induce a splitting for $\HN_{0}(\A)$, we get a commutative diagram,
 $$\xymatrix{&&&\HH_{0}(\A)\ar[ddd]^{=}\\
 \K_{0}(\A)\ar[rrru]^{\Ch}\ar[r]\ar[d]&\HN_{0}(\A)\ar[rru]_{\pi}\ar[r]^{\cong}_{\H}\ar[d]^{\j}&\oplus_{i\leq 0} \HH_{2\i}(\A)\ar[d]\ar[ur]\ar[ru]_{\mathsf{Pr}}&\\
 \K^{\top}_{0}(\A)\ar[r]&\HC^{\per}_{0}(\A)\ar[r]^{\cong}_{\H}&\oplus_{\i} \HH_{2\i}(\A)\ar[rd]^{\mathsf{Pr}}&\\
 &&&\HH_{0}(\A)}$$
 Note that the projection $\mathsf{Pr}\circ\H: \HN_{0}(\A)\rightarrow \HH_{0}(\A)$ is naturally the morphism $\pi$. The Hodge classes defined in \cite{perry2020integral} is isomorphic to the image $pr\circ\H\circ(\Ch^{\top}(\K^{\top}_{0}(\A)_{\Q})\cap \j(\HN_{0}(\A)))$ in $\HH_{0}(\A)$. By the commutative diagram, it is exactly the classes $\pi(\j^{-1}(\Ch^{\top}(\K^{\top}_{0}(\A)_{\Q})))\subset \HH_{0}(\A)$.
\end{proof}

\begin{lem}
Let $\A$ be a smooth proper $\3\dg$ category, the noncommutative Hodge-to de Rham spectral sequence degenerates \cite{kaledin2016spectral}.
\end{lem}

\begin{defn}\label{smpconj}
(Hodge conjecture for smooth proper $\3\dg$ categories) Define the Hodge classes in $\HH_{0}(\A)$ as $pr\circ\H(\Ch^{\top}(\K^{\top}_{0}(\A)_{\Q})\cap \j(\HN_{0}(\A)))$. Then the Hodge conjecture is that the Chern character $\Ch: \K_{0}(\A)\rightarrow \HH_{0}(\A)$ maps $\K_{0}(\A)_{\Q}$ surjectively into the Hodge classes.
\end{defn}

\begin{rem}\label{remconj}
 This is equivalent to the conjecture \ref{Mainconj} by the same argument in the proof of Theorem \ref{admissible}. We formulate this version because of it is Hodge original.
\end{rem}

\par

\begin{thm}\label{NHodge}
 Let $\X$ be a smooth projective variety. Hodge conjecture for $\X$ $\Leftrightarrow$ Noncommutative Hodge conjecture for $\Per_{\3\dg}(\X)$.
\end{thm}

\begin{proof}
 The commutative Hodge conjecture claims that the Chern character $\Ch \colon \K_{0}(\X)_{\Q}\longrightarrow \bigoplus_{\mathsf{p}}\H^{\mathsf{p},\mathsf{p}}(\X,\C)$ maps $\K_{0}(\X)_{\Q}$ surjectively to the rational Hodge classes. The noncommutative Hodge conjecture claims that the map $\Ch_{\Q}\colon \K_{0}(\X)_{\Q}=\K_{0}(\Per_{\3\dg}(\X))_{\Q}\longrightarrow \Hodge(\Per_{\3\dg}(\X))$ is surjective.

 \par

 There is a commutative diagram$\colon$
 $$\xymatrix{\HH_{0}(\Per_{\3\dg}(\X))\ar[ddd]^{\cong}&&\K^{\top}_{0}(\Per_{\3\dg}(\X))\ar[rd]\ar@/_/[ddd]&\\
 &\K_{0}(\Per_{\3\dg}(\X))\ar[ru]\ar[lu]^{\Ch}\ar[r]^{\Ch}\ar[d]^{\cong}&\HN_{0}(\Per_{\3\dg}(\X))\ar[llu]^{\pi}\ar[r]^{\hookrightarrow}\ar[d]^{\cong}&\HC_{0}^{\per}(\Per_{\3\dg}(\X))\ar[d]^{\cong}\\
 &\K_{0}(\X)\ar[rd]\ar[ld]_{\Ch}\ar[r]^{\Ch}&\bigoplus_{\i\leq 0} \H^{\mathsf{p},\mathsf{p}-2\i}(\X,\C)\ar[lld]^{\pi}\ar[r]^{\hookrightarrow}&\H^{\mathsf{even}}_{\mathsf{dR}}(\X,\C)\\
 \oplus_{\p}\H^{\p,\p}(\X,\C)&&\K^{\top}_{0}(\X)\ar[ru]&}$$
 We explain the commutative diagram. There is a
 natural quasi isomorphism of double complexes of periodic cyclic homology $\mathsf{Tot}^{\bullet,\bullet}(\Per_{\3\dg}(\X))\rightarrow  \mathsf{Tot}^{\bullet,\bullet}(\R\Gamma(\oplus \Omega^{\i}_{\X}[\i]))$ which is described by B.\ Keller in \cite{Keller1998}.

 After identifying $\HC_{0}^{\per}(\Per_{\3\dg}(\X))$ with $\H_{\mathsf{dR}}^{\mathsf{even}}(\X,\C)$, the noncommutative Chern character becomes the usual Chern character. The reader can refer to C.\ Weibel \cite[Proposition 3.8.1]{K-theory/0046} or \cite[Proposition 4.32]{blanc_2016}.  Hence, the noncommutative Chern character maps $\K_{0}(\X)_{\Q}$ surjectively to the noncommutative rational Hodge classes if and only if the commutative Chern character maps $\K_{0}(\X)_{\Q}$ surjectively to the commutative rational Hodge classes.
\end{proof}

\begin{thm}\label{MoritaHodge}
Suppose $\F\colon \A \longrightarrow  \B$ is a derived Morita equivalence, then Hodge conjecture is true for $\A$ if and only if it is true for $\B$.
\end{thm}

\begin{proof}
 The topological and algebraic $\K$-theory, Hochschild homology, periodic (negative) cyclic homology are all additive invariants. We have a commutative diagram,
 $$\xymatrix{ \K^{\top}_{0}(\A)\ar[rddd]\ar[rrr]^{\cong}&&&\K^{\top}_{0}(\B)\ar[lddd]\\
&\K_{0}(\A)\ar[lddd]_{\Ch}\ar[r]^{\cong}\ar[d]\ar[lu]&\K_{0}(\B)\ar[rddd]^{\Ch}\ar[d]\ar[ru]&\\
 &\HN_{0}(\A)\ar[ldd]^{\pi}\ar[r]^{\cong}\ar[d]&\HN_{0}(\B)\ar[d]\ar[rdd]_{\pi}&\\
& \HC_{0}^{\per}(\A)\ar[r]^{\cong}&\HC_{0}^{\per}(\B)&\\
\HH_{0}(\A)\ar[rrr]^{\cong}&&&\HH_{0}(\B)}$$
 whose rows are isomorphisms.
 It is clear that any morphism of dg categories induce a morphism of Hodge classes: write $\phi$ as the corresponding morphism form additive invariants of $\A$ to $\B$. Let $x\in \Hodge(\A)$, this implies that there is $ \x1\in \HN_{0}(\A)$ such that $\pi(\x1)=\mathsf{x}$, and $\y\in \K_{0}^{\top}(\A)_{\Q}$ such that $\j(\x1)=\Ch_{\Q}^{\top}(\y)$. Apply $\phi$, we get $\phi(\mathsf{x})=\pi(\phi(\x1))$, and $\j(\phi(\x1))=\Ch_{\Q}^{\top}(\phi(\y))$, that is, $\phi(\mathsf{x})\in \Hodge(\B)$.
 There is a commutative diagram.

$$\xymatrix{\K_{0}(\A)\ar[r]^{\cong}\ar[d]^{\Ch}&\K_{0}(\B)\ar[d]^{\Ch}\\
\Hodge(\A)\ar[r]^{\cong}&\Hodge(\B)}$$
The isomorphism of Hodge classes is as follows: Take $\mathsf{z}\in \Hodge(\B)$, since $\Phi$ induces isomorphis $\HH_{0}(\A)\cong \HH_{0}(\B)$, there exist unique $\mathsf{x}\in \HH_{0}(\A)$ such that $\phi(\mathsf{x})=\mathsf{z}$. It can be shown that $\mathsf{x}\in \Hodge(\A)$ by diagram chasing.
\end{proof}

\begin{cor}\label{Uniqueenhanced}
 For the unique enhanced triangulated categories, we can define its Hodge conjecture via its smooth and proper $\3\dg$ enhancement (if it exists). The Hodge conjecture does not depend on the $\3\dg$ enhancement.
\end{cor}

\begin{proof}
This is because two $\3\dg$ enhancements of the unique enhanced triangulated categories are connected by a chain of quasi-equivalences, and the corollary follows from Theorem \ref{MoritaHodge}.
\end{proof}

\begin{rem}
For a projective smooth variety $\X$, $\b\X\cong \Perf(\X)$ is a unique enhanced triangulated category. Thus, it suffices to check whether the conjecture is true for any pre-triangulated $\3\dg$ enhancement of $\b\X$.
\end{rem}

\begin{thm}\label{SODHodge}
 Suppose we have a $\SOD$, $\b\X=\langle \A,\B \rangle$. There are natural $\3\dg$ enhancement $\A_{\3\dg}$, $\B_{\3\dg}$ of \mbox{$\A$, $\B$} corresponding to $\3\dg$ enhancement $\Per_{\3\dg}(\X)$ of $\b\X$.
 $$\text{\it Hodge conjecture for}\ \X\ \Leftrightarrow \text{\it Noncommutative Hodge conjecture for}\ \A_{\3\dg}\ \text{\it and}\ \B_{\3\dg}.$$
\end{thm}

\begin{proof}
 We still write $\A$ and $\B$ as dg categories corresponding to the natural $\3\dg$ enhancement again. We can lift the semi-orthogonal decomposition to the $\3\dg$ world by \cite[Proposition 4.10]{Kuznetsov2015CategoricalRO}. That is, there is a diagram
 $$\xymatrix{\B\ar[r]_{\i} &\mathsf{D}\ar[r]_{\L}\ar@/_/[l]_{\R}&\A\ar@/_/[l]_{\j}
}$$
where $\mathsf{D}$ is certain gluing of $\A$ and $\B$ and it is quasi-equivalent to $\Per_{\3\dg}(\X)$. Therefore, we still have a diagram such that $\i+\j$ induces isomorphism of $\K$ group, and $\i_{\H}+\j_{\H}$ induces
$$\xymatrix{\K_{0}(\B)\ar[r]_{\i}\ar[d]_{\Ch}&\K_{0}(\mathsf{D})\ar[r]_{\L}\ar@/_/[l]_{\R}\ar[d]_{\Ch}&\K_{0}(\A)\ar[d]_{\Ch}\ar@/_/[l]_{\j}\\
\Hodge(\B)\ar[r]_{\i_{\H}}&\Hodge(\mathsf{D})\ar[r]_{\L_{\H}}\ar@/_/[l]_{\R_{\H}}&\Hodge(\A)\ar@/_/[l]_{\j_{\H}}}$$
Hence $\Ch_{\mathsf{D},\Q}$ maps $\K_{0}(\d)_{\Q}$ surjectively to $\Hodge(\mathsf{D})$ if and only if $\Ch_{\B,\Q}$ and $\Ch_{\A,\Q}$ map $\K_{0}(\B)_{\Q}$ and $\K_{0}(\A)_{\Q}$ surjectively to $\Hodge(\B)$ and $\Hodge(\A)$ respectively. But the noncommutative Hodge conjecture is true for $\mathsf{D}$ if and only if it is true for the Hodge conjecture of $\X$ by the Theorem \ref{NHodge} and Theorem \ref{MoritaHodge}. Thus, the statement follows.
\end{proof}

\begin{rem}
Similar to the geometric case \ref{GSODHodge1}, the statement is still true if there are more than two components for $\SOD$s.
\end{rem}

\par

\begin{thm}
Let $\A$ be a admissible subcategories of $\b\X$ where $X$ is a smooth projective smooth variety.
\end{thm}

We immediately reprove Theorem \ref{GSODHodge1}.

\begin{cor}\label{GSODHodge2}
Let $\X$ be a projective smooth variety, suppose there is a $\SOD$, $\b\X= \langle \2\b\Z, \1\b\Y\rangle$. Then Hodge conjecture is true for $\X$ if and only for $\z$ and $\Y$. In particular Hodge conjecture is a derived invariant.
\end{cor}

\begin{proof}
  According to Theorem \ref{SODHodge}, Hodge conjecture is true for $\X$ if and only if it is true for corresponding $\3\dg$ enhancement of $\2\b\Z$ and $\1\b\Y$. Since $\2\b\Z$ and $\1\b\Y$ are unique enhanced triangulated categories \cite{Lunts_2010}, hence the Hodge conjecture is true for $\X$ if and only for $\z$ and $\Y$.
\end{proof}

\begin{cor}\label{Orlovformula}
 Consider blow up $\X$ of $\Y$ with smooth center $\z$, according to Orlov's blow-up formula \cite[Theorem 4.2]{Bondal2002DerivedCO}, we have a $\SOD$, $\b\X=\langle \2\b\Z,\cdots, \2\b\Z, \1\b\Y\rangle$. Hence the Hodge conjecture is true for $\X$ if and only if for $\z$ and $\Y$.
\end{cor}

\begin{rem}
  It was known by classical method. We can even write down the $\Chow$ groups with respect to the blow up, for explicit details, the reader can refer to the book of C.\ Voisin, ``Hodge theory and complex algebraic geometry $\uppercase\expandafter{\romannumeral2}$'' \cite[Theorem 9.27]{voisin_2003}
\end{rem}

\begin{cor}\label{EC2}
 We reprove Corollary \ref{EC}: Suppose $\b\X$ admits a full exceptional collection, then the Hodge conjecture is true for $\X$.
\end{cor}

For low dimensional varieties, Hodge conjecture is a birational invariant. We use the following lemma$\colon$

\begin{lem}\label{zigzagBlowu}
(\cite[Theorem 0.1.1]{Abramovich1999TorificationAF}) Let $\X$ and $\Y$ be proper smooth varieties. If $\X$ is birational to $\Y$, then there is a chain of blow-ups and blow-downs of smooth centers connecting $\X$ and $\Y$.
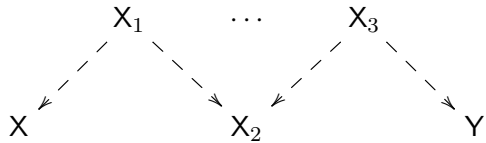
$$\xymatrix{&\X_{1}\ar@{-->}[ld]\ar@{-->}[rd]&\cdots&\X_{3}\ar@{-->}[ld]\ar@{-->}[rd]&\\
\X&&\X_{2}&&\Y}$$
\end{lem}

The following may be well known for the expects, see also \cite{meng2019hodge}. Here, we use the noncommutative techniques to reprove the results.

\begin{thm}\label{4fouldRinvariant}
Since Hodge conjecture is true for $0$, $1$, $2$ and $3$ dimensional varieties, the Hodge conjecture is a birational invariant for $4$ and $5$ dimensional varieties.
\end{thm}

\begin{proof}
Combining Corollary \ref{Orlovformula} and Lemma \ref{zigzagBlowu}, and observe that $\X$ and $\Y$ are connected by a chain of blow-ups of smooth center whose dimension is less or equal to $3$.
\end{proof}

\subsection{Application to geometry and examples}\label{section4.2}
\par
The survey ``Noncommutative counterparts of celebrated conjecture'' \cite[Section 2]{tabuada2019noncommutative} provides many examples of the applications to the geometry for some celebrated conjectures. The examples also apply to the
noncommutative Hodge conjecture. In this subsection, we still show some interesting examples.

\par

There is a universal functor
$$\U \colon \dgcat \longrightarrow \NChow.$$
We call $\U(\A)$ the noncommutative $\Chow$ motive corresponds to $\A$. We write the image of $\U(\A)$ in $\NChow_{\Q}$ as $\U(\A)_{\Q}$. Similar to works of  G.\ Tabuada, the noncommutative Hodge conjecture is compatible with the direct sum decomposition of the noncommutative $\Chow$ motives.

\begin{thm}\label{SODNMotive}
 Let $\A$, $\B$ and $\c$ be smooth and proper $\3\dg$ categories. Suppose there is a direct sum decomposition$\colon$
 $\U(\c)_{\Q}\cong \U(\A)_{\Q}\oplus \U(\B)_{\Q}$, then noncommutative Hodge conjecture holds for $\c$ if and only if it holds for $\A$ and $\B$.
\end{thm}

\begin{proof}
 This follows from the fact that the periodic (negative) cyclic homology and rational (topological or algebraic) $\K$-theory are all additive invariants, and the corresponding target categories are idempotent complete. The proof is similar to Theorem \ref{SODHodge}.
\end{proof}

\begin{eg}
 Suppose we have a semi-orthogonal decomposition$\colon$$\H^{0}(\c)=\langle \H^{0}(\A),\H^{0}(\B)\rangle$, then
 $\U(\c)\cong \U(\A)\oplus \U(\B)$.
\end{eg}

\subsubsection{Fractional Calabi--Yau categories}
\begin{thm}\label{egCY}(\cite[Theorem 3.5]{Kuznetsov_2019}) Let $\X$ be a hypersurface of degree $\leq \n+1$ in $\mathbb{P}^{\n}$. There is a semi-orthogonal decomposition $\colon$
$$\Perf(\X)=\langle \mathcal{T}(\X),\O_{\X},\cdots,\O_{\X}(\n-\mathsf{deg}(\X))\rangle.$$
$\mathcal{T}(\X)$ is a fractional Calabi--Yau of dimension $\frac{(\n+1)(\mathsf{deg}(\X)-2)}{\mathsf{deg}(\X)}$.
Then $$\U(\X)\cong \U(\mathcal{T}_{\3\dg}(\X))\oplus \U(\k)\oplus\cdots \oplus \U(\k).$$ Therefore, Hodge conjecture of $\X$ $\Leftrightarrow$ Noncommutative Hodge conjecture of $\mathcal{T}_{\3\dg}(\X)$.
\end{thm}

\subsubsection{Twisted scheme.}
\begin{defn}
 Let $\X$ be a scheme with structure sheaf $\O_{\X}$. $\A$ is a sheaf of Azumaya algebra over $\X$. We call the derived category of perfect $\A$ module $\Perf(\X,\A)$ the twisted scheme.
\end{defn}

\begin{thm}\label{egTwistedscheme1}
  Noncommutative Hodge conjecture for $\Per_{\3\dg}(\X,\A)$ $\Leftrightarrow$ Noncommutative Hodge conjecture for $\Per_{\3\dg}(\X)$.
 \end{thm}

 \begin{proof}
 According to \cite[Theorem 2.1]{tabuadavandenbergh2015}, $\U(\Per_{\3\dg}(\X,\A))_{\Q}\cong \U(\Per_{\3\dg}(\X))_{\Q}$. Thus, by Theorem \ref{SODNMotive}, the statement follows.
 \end{proof}

 \subsubsection{Cubic fourfold containing a plane.}
 \begin{eg}\label{egTwistedscheme2}
 Let $\X$ be a cubic fourfold containing a plane. There is a semi-orthogonal decomposition\cite[Theorem 4.3]{Kuznetsov2010DerivedCO}
 $$\Perf(\X)=\langle\Perf(\S,\A),\O_{\X},\O_{\X}(1),\O_{\X}(2)\rangle.$$
 $\mathsf{S}$ is a $\K_{3}$ surface, and $\A$ is a sheaf of Azumaya algebra over $\mathsf{S}$. Since the noncommutative Hodge conjecture is true for $\Per_{\3\dg}(\mathsf{S},\A)$ which is unique enhanced, hence the Hodge conjecture is true for $\X$.
 \end{eg}

 \subsubsection{Quadratic fibration.}
 \begin{eg}\label{egQF}
   Let $\f\colon \X\longrightarrow \mathsf{S}$ be a smooth quadratic
   fibration, for example, the smooth quadric in relative projective space $\mathbb{P}^{\n}_{\mathsf{S}}$. There is a semi-orthogonal decomposition
   $$\Perf(\X)=\langle \Perf(\mathsf{S},\mathsf{Cl}_{0}),\Perf(\mathsf{S}),\cdots,\Perf(\mathsf{S})\rangle.$$
   $\mathsf{Cl}_{0}$ is a sheaf of Azumaya algebra over $\mathsf{S}$ if the dimension $\n$ of the fiber of $\f$ is odd \cite{Kuznetsov2005DerivedCO}. Thus, the Hodge conjecture of $\X$ $\Leftrightarrow$ $\mathsf{S}$. Moreover, if $\dim \mathsf{S}\leq 3$, the Hodge conjecture for $\X$ is true.
 \end{eg}

 \subsubsection{HP duality}\

 Let $\X$ be a projective smooth variety with morphism $\f\colon \X\longrightarrow \mathbb{P}(\V)$. Set $\O_{\X}(1)=\f^{*}\O_{\mathbb{P}(\V)}(1)$. Assume there is a $\SOD$
 $$\b\X=\langle \A_{0},\A_{1}(1),\cdots,\A_{\m-1}(\m-1)\rangle$$
 where $\A_{\m-1}\subset \cdots \subset\A_{1}\subset \A_{0}$.
 Define $\H:= \X\times _{\mathbb{P}(\V)}\mathsf{Q}$, where $\mathsf{Q}$ is the incidence quadric in $\mathbb{P}(\V)\times \mathbb{P}(\V^{\ast})$. Then, there is a $\SOD$
 $$\mathsf{D}^{\mathsf{b}}(\mathsf{H})=\langle \mathcal{L},\A_{1,\mathbb{P}(\V^{\ast})}(1),\cdots,\A_{\m-1,\mathbb{P}(\V^{\ast})}(m-1)\rangle.$$
 Projective smooth variety $\Y$ with morphism $g:\Y\longrightarrow \mathbb{P}(\V^{\ast})$ is called homological projective dual of $\X$ if there is an object $\mathcal{E}\in \mathsf{D}^{\mathsf{b}}(\mathsf{H}\times_{\mathbb{P}(\V^{\ast})}\Y)$ which induces an equivalence from $\1\b\Y$ into~$\mathcal{L}$.

\par

We refer to \cite[Section 2.3]{kuznetsov2015semiorthogonal} or Kuznetsov's original paper \cite{PMIHES_2007__105__157_0}. Let $(\Y,\mathsf{g})$ be a $\mathsf{HP}$ dual of $(\X,\f)$, then \\
1. There is a $\SOD$
$$\1\b\Y=\langle \B_{\n-1}(1-\n),\cdots,\B_{1}(-1),\B_{0}\rangle$$
where $\B_{\n-1}\subset \cdots \subset \B_{1}\subset \B_{0}$. Moreover $\A_{0}\cong \B_{0}$ via Fourier-Mukai functor. \\
2. (Symmetry) $(\X,\f)$ is a $\mathsf{HP}$ dual of $(\Y,\mathsf{g})$. \\
3. For any subspace $\L\subset \V^{\ast}$, define $\X_{\L}=\X\times_{\mathbb{P}(\V)}\mathbb{P}(\L^{\perp})$ and $\Y_{\L}= \Y\times _{\mathbb{P}(\V^{\ast})}\mathbb{P}(\L)$. If we assume that they have the expected dimension,
$\dim\X_{\L}=\dim\X-\dim \L$, $\dim\Y_{\L}=\dim\Y-(\dim \V-\dim \L)$, and write $\dim \L=\r$, $\dim \V=\mathsf{N}$,
then there are $\SOD$ such that $\mathcal{L}_{\X,\L}\cong \mathcal{L}_{\Y,\L}$.
$$\mathsf{D}^{\mathsf{b}}(\X_{\L})=\langle \mathcal{L}_{\X,\L},\A_{\r}(\r),\cdots,\A_{\m-1}(\m-1)\rangle.$$
$$\mathsf{D}^{\mathsf{b}}(\Y_{\L})=\langle \B_{\n-1}(1-\n),\cdots, \B_{\mathsf{N}-\mathsf{r}}(\mathsf{r}-\mathsf{N}),\mathcal{L}_{\Y,\L}\rangle.$$
\begin{thm}\label{HPD}
We write $\Hodge(\bullet)$ if the (noncommutative) Hodge conjecture is true for varieties (smooth and proper $\3\dg$ categories). Then, $\Hodge(\X)$ $\Leftrightarrow$ $\Hodge(\A_{0})$ $\Leftrightarrow$ $\Hodge(\B_{0})$ $\Leftrightarrow$ $\Hodge(\Y)$. If we assume $\Hodge(\X)$, then $\Hodge(\X_{\L})\Leftrightarrow \Hodge(\Y_{\L})$.
\end{thm}

\begin{proof}
  The midterm equivalence $\Hodge(\A_{0})\Leftrightarrow\Hodge(\B_{0})$ is because $\A_{0}\cong\B_{0}$ via a Fourier-Mukai functor, and then there is an isomorphism of natural $\3\dg$ enhancements $\A_{\3\dg,0}\cong\B_{\3\dg,0}$ in $\Hmo$, see a proof in \cite[Section 9]{bernardara2014semiorthogonal}. Since $\mathcal{L}_{\X,\L}\cong \mathcal{L}_{\Y,\L}$ via Fourier-Mukai functor, the statement $\Hodge(\X_{\L})\Leftrightarrow \Hodge(\Y_{\L})$ follows from the same argument.
\end{proof}

\begin{rem}
 The $\HPD$ can be generalized to the noncommutative version, see the discussion in \cite[Section 3.4]{kuznetsov2015semiorthogonal} or the paper by Alexander Perry, ``Noncommutative homological projective duality'' \cite{PERRY2019877}.
\end{rem}

\begin{eg}\label{egHPD}
One of the nontrivial examples of the Homological projective duality comes from the Grassmannian-Pfaffian duality. Let $\W$ be a dimension $\n$ vector space, $\X=\mathsf{Gr}(2,\W)$ the Grassmannian of 2-dimensional sub-vector spaces of $\W$.
Consider the projective space $\mathbb{P}(\wedge^{2}\W^{\ast})$, there is a natural filtration called the Pfaffian filtration$\colon$ $\Pf(2,\W^{\ast})\subset \Pf(4,\W^{\ast})\cdots \subset \mathbb{P}(\wedge^{2}\W^{\ast})$.
$$\Pf(2\k,\W^{\ast})=\{\omega\in\mathbb{P}(\wedge^{2}\W^{\ast})\mid\mathsf{rank}(\omega)\leq 2\k\}$$
The intermediate Pfaffians are no longer smooth but with singularities. The singularity of $\Pf(2\k,\W^{\ast})$ is $\Pf(2\k-2,\W^{\ast})$. Classically, it was known that $\Y=\Pf(2\lfloor\frac{\n}{2}\rfloor-2,\W^{\ast})$ is the classical projective dual of $\X=\mathsf{Gr}(2,\W)$ via the Pl\"{u}cker embedding. For $\n\leq 7$, the noncommutative categorical resolution of $\Pf(2\lfloor\frac{\n}{2}\rfloor-2,\W^{\ast})$ is the homological projective dual of $\mathsf{Gr}(2,\W)$. However, it was not known for the cases $\n\geq 8$. The interested reader can refer to a survey \cite[Section 4.4, Conjecture 4.4]{kuznetsov2015semiorthogonal} or Kuznetsov's original paper \cite{Kuznetsov2006HomologicalPD}.

\par
The known nontrivial Grassmannian-Pfaffian duality are the cases $\n=6, 7$. In these cases, Hodge conjecture is true for $\X$ since it has a full exceptional collection, then the noncommutative Hodge conjecture is true for the noncommutative categorical resolution of the Pfaffians. However, the Hodge conjecture is trivial for the noncommutative category since it automatically has full exceptional collections, or the geometric resolution of the Pfaffians are of the form $\mathbb{P}_{\mathsf{Gr}(2,\W)}(\E)$ \cite[Section 4]{Kuznetsov2006HomologicalPD} for some vector bundle $\E$. It has a full exceptional collection too.

\par

 We expect to obtain duality of the Hodge conjecture for $\X_{\L}$ and $\Y_{\L}$ when they are smooth, and have the expected dimension.
 According to the Lefschetz hyperplane theorem, there is a commutative diagram for $\i\leq \dim \X_{\L}-1\colon$
$$\xymatrix{\CH^{\i}(\X_{\L})_{\Q}\ar[r]&\H^{\i}(\X_{\L},\mathbb{Q})\\
\CH^{\i}(\mathsf{Gr}(2,\W))_{\Q}\ar[r]^{\cong}\ar[u]&\H^{\i}(\mathsf{Gr}(2,\W),\mathbb{Q})\ar[u]^{\cong}}$$
The Hodge conjecture is true for weight less than $\dim\X_{\L}$. By the hard Lefschetz isomorphism, it is still true for weight greater than $\dim\X_{\L}$. Thus, if $\dim\X_{\L}$ is odd, the Hodge conjecture for $\X_{\L}$ is true.

\par
The following examples for $\n=6,7$ are from paper \cite[Section 10]{Kuznetsov2006HomologicalPD}.
\par
$\uppercase\expandafter{\romannumeral1}$.\ $\n=6$, $\dim\X_{\L}=8-\dim\L$, $\dim\Y_{\L}=\dim\L-2$. When $\dim\L=6$, the expected dimension of $\X_{\L}$ is $2$ while the expected dimension of $\Y_{\L}$ is $4$. This is the duality between Pfaffian cubic fourfold and the $\K_{3}$ surface \cite{Kuznetsov2006HomologicalPD}. When $\dim\L=5$, $\dim\X_{\L}=\dim\Y_{\L}=3$, the Hodge conjecture is true by dimension reason. When $\dim\L=4$, $\Y_{\L}=\Pf(4,6)\cap \mathbb{P}^{3}$ is a cubic surface. Then $\X_{\L}=\mathsf{Gr}(2,6)\cap \mathbb{P}^{10}$ has a full exceptional collection. $\X_{\L}$ is a rational Fano $4$-fold \cite[Section 2.2, Theorem 2.2.1]{XF}. Hence, the Hodge conjecture is true for $\X_{\L}$ by weak factorization theorem \cite[Theorem 0.1.1]{Abramovich1999TorificationAF}. When $\dim\L=3$, $\dim\X_{\L}=5$, the Hodge conjecture is true for $\X_{\L}$. When $\dim\L=2$, $\X_{\L}$ admits a full exceptional collection. We obtain a table.\\
\begin{center}
\begin{tabular}{|c|c|c|c|}
  \hline

  $\dim\L$ & $\dim\X_{\L}$&$\dim\Y_{\L}$&classically \\
  \hline
  2&6&0& \\
  \hline
   3& 5&1&Known\\
  \hline
   4&4&2& Known,\ $\X_{\L}$\ is\ a\ rational\ Fano\ 4-fold  \\
  \hline
   5&3&3&Known,\ they\ are\ 3-fold\\
  \hline
   6&2&4&Known,\ $\Y_{\L}$\ is\ a\ cubic\ 4-fold\\
  \hline
\end{tabular}
\end{center}

$\uppercase\expandafter{\romannumeral2}$.\ $\n=7$, $\dim\X_{\L}=10-\dim\L$, $\dim\Y_{\L}=\dim\L-4$. For example, take $\dim\L=7$. The expected dimension of $\X_{\L}$ and $\Y_{\L}$ are both $3$. The Hodge conjecture is true for them by dimension reason. When $\dim\L=5$, $\dim\X_{\L}=5$, the Hodge conjecture is true for $\X_{\L}$.
When $\dim\L=6$, $\dim \X_{\L}=4$, it is a fano $4$-fold. When $\dim\L=8$, $\dim\Y_{\L}=4$, it is a fano $4$-fold. Since fano varieties are uniruled, the Hodge conjecture is true for fano $4$-folds \cite{CM78}. When $\dim\Y_{\L}=9$, $\Y_{\L}$ is a fano $5$-fold, the Hodge conjecture is true for fano 5-folds by \cite{DA}. When $\dim\L=10$, $\Y_{\L}$ admits a full exceptional collection. We obtain a table.

\begin{center}
\begin{tabular}{|c|c|c|c|}
  \hline

  $\dim\L$ & $\dim\X_{\L}$&$\dim\Y_{\L}$&classically \\
  \hline
   5& 5&1&Known,\ since\ dimension\ of\ $\X_{\L}$\ is\ odd\\
  \hline
   6&4&2&Known,\ $\X_{\L}$\ is\ a\ fano\ 4-fold\\
  \hline
   7&3&3&Known\ by\ dimension\ reason\\
  \hline
   8&2&4&Known,\ $\Y_{\L}$\ is\ a\ fano\ 4-fold\\
  \hline
   9&1&5&Known,\ $\Y_{\L}$\ is\ a\ fano\ 5-fold \\
  \hline
  10&0&6& \\
  \hline
\end{tabular}
\end{center}

\begin{rem}
  We thanks Claire Voisin pointing out to the author a classical result that the Hodge conjecture is true for uniruled $4$-folds \cite{CM78}. Even though most examples here can be proved by classical methods, we hope that we can use geometry of dual varieties to prove Hodge conjecture of these examples, see also the Conjecture \ref{Conjprojectivedual} below. We leave the blanks in the tables since it is not known for the author whether the Hodge conjecture is proved for these cases previously.
\end{rem}

$\uppercase\expandafter{\romannumeral3}$.\ For $\n\geq 8$, the $\mathsf{HPD}$ is not constructed. However, when $\n=10$, there is an interesting picture inspired by the Mirror Symmetry which was constructed by E.\ Segal and RP.\ Thomas \cite[Theorem A]{Segal:2014jua}.
 \par
 Let $\L$ be a $5$-dimensional subspace of $\wedge^{2}\W^{\ast}$, $\L^{\perp}\subset \wedge^{2}\W$. Write $\X=\mathsf{Gr}(2,10)\subset \mathbb{P}^{44}$ and $\Y=\Pf(8,10)\subset\mathbb{P}^{44}$; $\X_{\L}=\mathbb{P}(\L^{\perp})\cap\X$, $\Y_{\L}=\mathbb{P}(\L)\cap\Y$. We choose general linear subspace $\L$ such that both $\X_{\L}$ and $\Y_{\L}$ are smooth. In particular, $\Y_{\L}$ is quintic $3$-fold and $\X_{\L}$ is a Fano $11$-fold. According to E.\ Segal and RP.\ Thomas \cite[Theorem A]{Segal:2014jua}, there is a fully faithful embedding
 $$\mathsf{D}^{\mathsf{b}}(\Y_{\L})\hookrightarrow \mathsf{D}^{\mathsf{b}}(\X_{\L}).$$
Let $\A$ be the exceptional collections $\{\Sym^{3}\S, \Sym^{2}\S,\S,\O\}$ of $\mathsf{D}^{\mathsf{b}}(\mathsf{Gr}(2,10))$, where $\S$ is the tautological bundle on $\mathsf{Gr}(2,10)$. It restricts to an exceptional collections in $\mathsf{D}^{\mathsf{\mathsf{b}}}(\X_{\L})$ by techniques in \cite{Kuznetsov2006HomologicalPD}. Then, let $\langle \A,\A(1),\cdots,\A(4)\rangle$ be an exceptional collection in $\mathsf{D}^{\mathsf{b}}(\X_{\L})$. They are right orthogonal to the above embedding
of $\mathsf{D}^{\mathsf{b}}(\Y_{\L})$, see description in \cite[Remark 3.8]{Segal:2014jua}. The Hochschild homology $\HH_{0}(\X_{\L})\cong \mathbb{C}^{24}$ and $\HH_{0}(\Y_{\L})\cong \mathbb{C}^{4}$. Therefore, $0^\text{th}$ Hochschild homology of the right orthogonal complement of $\langle \A,\A(1),\cdots,\A(4),\mathsf{D}^{\mathsf{b}}(\Y_{\L})\rangle$ is trivial. Thus, the Hodge conjecture for $\X_{\L}$ follows from the additive theory.

\par

Inspired by the examples above, we expect that even though we do not have $\HPD$, the duality of the Hodge conjecture between linear section of the dual varieties can be proved by classical methods.
\end{eg}

\begin{conj}\label{Conjprojectivedual}
 Let $\X\subset\mathbb{P}(\V)$ be a projective smooth variety. Suppose the Hodge conjecture is true for $\X$. Let $\Y\subset\mathbb{P}(\V^{\ast})$ be the projective dual of $\X\subset\mathbb{P}(V)$. Choose a linear subspace $\L\subset\V^{\ast}$.
  Suppose the linear sections $\X_{\L}=\X\cap \mathbb{P}(\L^{\perp})$ and $\Y_{\L}=\Y\cap \mathbb{P}(\L)$ are both of expected dimension and smooth. Then, the Hodge conjecture of $\X_{\L}$ is equivalent to the Hodge conjecture of $\Y_{\L}$.
\end{conj}

\subsection{Connective dg algebras}
In this subsection, we prove that the noncommutative Hodge conjecture is true for the connective $\3\dg$ algebras.

\begin{defn}
 $\a$ is called a connective $\3\dg$ algebra if $\H^{\i}(\a)=0$ for
 $\i> 0 $.
\end{defn}

\begin{thm}\label{propersmoothconectivealgebraHodge}
 If $\a$ is a smooth and proper connective $\3\dg$ algebra, the noncommutative Hodge conjecture is true for $\a$.
\end{thm}

\begin{proof}
 According to recent work of Theo Raedschelders and Greg Stevenson \cite[Corollary 4.3, Theorem 4.6]{raedschelders2020proper}, $\U(\a)_{\Q}\cong\U(\H^{0}(\a)/\mathsf{Jac}(\H^{0}(\a)))_{\Q}\cong \oplus \U(\C)_{\Q}$. Hence, the noncommutative Hodge conjecture is true for connective $\3\dg$ algebras. In particular, it is true for the proper smooth algebras (concentrated in degree 0).
\end{proof}

\par

 We provide another proof which involves more calculation for smooth and proper algebras.
Clearly, proper algebras are finite dimensional algebras. Due to R.\ Rouquier \cite[section 7]{rouquier_2008}, $\mathsf{Pdim}_{\a^{\mathsf{e}}}(\a)= \mathsf{Pdim}(\a)$, smooth algebras are finite global dimensional algebras.
Consider the acyclic quiver $\mathsf{Q}$ with finitely many vertices. Let $\a:= \mathsf{kQ/I}$ be the quiver algebra with relations, where $\mathsf{kQ}$ is the path algebra of $\mathsf{Q}$. Then, $\a$ is a smooth and proper algebra. The noncommutative Hodge conjecture is true for $\a$.

\begin{thm}\label{algebra}
Let $\a= \mathsf{kQ/I}$. Consider natural Chern character map
 $$\Ch\colon \K_{0}(\a)\longrightarrow \HH_{0}(\a).$$
 Then, $\Im\Ch_{\Q}\otimes \C= \HH_{0}(\a)$. In particular, the noncommutative Hodge conjecture is true for $\a$.
\end{thm}

\begin{proof}
 Firstly, for the algebra $\a$, $\HH_{0}(\a)\cong \a/[\a,\a]\cong \k\langle \mathsf{e}_{1},\mathsf{e}_{2},\cdots, \mathsf{e}_{\n}\rangle$ where $\mathsf{e}_{\i}$ is vertex of the quiver $\mathsf{Q}$. We write $\mathsf{S}_{\i}=\a\cdot \mathsf{e}_{\i}$ which is considered as a left $\a$ module, $[\mathsf{S}_{\i}]\in \K_{0}(\a)$. We prove that $\Ch([\mathsf{S}_{\i}])= \mathsf{e}_{\i}$.
 According to the paper of McCarthy, ``Cyclic homology of an exact category'' \cite[section 2]{MCCARTHY1994251}, there is an natural identification of Hochschild homology$\colon$
$$\bigoplus_{\n} \Hom_{\a}(\a,\a)\otimes \cdots \otimes\Hom_{\a}(\a,\a)\longrightarrow \bigoplus_{\X,\Y,\n} \Hom_{\a}(\X,\E_{1})\otimes\cdots \otimes \Hom_{\a}(\E_{\n},\Y).$$
It is a natural quasi-isomorphism, the left hand side is exactly the bar complexes of $\a$. $\X$ and $\Y$ are both projective left $\a$ modules. Under this identification, the image of the Chern character of object $[\P]$ that is projective $\a$ module is the homology class of $\mathsf{id}_{\P}$ in the right hand side complex. Consider the local picture$\colon$
$$\mathsf{Bar}\colon \Hom_{\a}(\mathsf{S}_{\i},\a)\otimes \Hom_{\a}(\a,
\mathsf{S}_{\i})\longrightarrow \Hom_{\a}(\mathsf{S}_{\i},\mathsf{S}_{\i})\oplus \Hom_{\a}(\a,\a).$$
Let $\f \in \Hom(\mathsf{S}_{\i},\a)$ be the natural inclusion, $\mathsf{e}_{\i}\in \Hom_{\a}(\a,\mathsf{S}_{\i})$ be the multiplication by $\mathsf{e}_{\i}$. Then $\mathsf{Bar}(\f\otimes \mathsf{e}_{\i})=\mathsf{id}_{\mathsf{S}_{\i}}-\mathsf{e}_{\i}$. Therefore, $[\mathsf{e}_{\i}]=[\mathsf{id}_{\mathsf{S}_{\i}}]$ in $\HH_{0}(\Proj\ \a)$. Hence $\Ch([\mathsf{S}_{\i}])= [\mathsf{e}_{\i}]$. Finally, $\Im\Ch_{\Q}\otimes \C= \HH_{0}(\a)$. Since $\Im\Ch_{\Q} \subset \HH_{0,\Q}(\a)$, therefore $\Im\Ch_{\Q} = \HH_{0,\Q}(\a)$.
\end{proof}

   A finite dimensional algebra $\a$ is (derived) Morita equivalent to an elementary algebra which is isomorphic to $\mathsf{kQ/I}$ for some quiver $\mathsf{Q}$. Clearly $\mathsf{kQ/I}$ is smooth and proper if $\a$ is smooth and proper. Then according to Theorem \ref{algebra}, the Hodge conjecture is true for any smooth and finite dimensional algebra $\a$.

\begin{rem}
 A.\ Perry pointed out to the author that if $\a$ is a smooth and proper algebra, $\Perf(\a)$ can be an admissible
 subcategory of the $\Perf(\X)$ which admits full exceptional collections for some smooth and projective varieties $\X$ by Orlov \cite[section 5.1]{Orlov2016SMOOTHAP}.
 Therefore, the noncommutative Hodge conjecture of $\a$ is true.
\end{rem}
\par
Classically, given any projective smooth variety $\X$, there is a compact generator $\E$ of $\mathsf{D}_{\mathsf{Qch}}(\X)$. Write $\E$ again after the resolution to an injective complex. Denote $\a=\Hom_{\3\dg}(\E,\E)$, then there is an equivalence $\mathsf{D}^{\per}(\a)\cong \Perf(\X)$ and chain of  derived Morita equivalences between $\Per_{\3\dg}(\a)$ and $\Per_{\3\dg}(\X)$. Thus, commutative Hodge for $\X$ $\Leftrightarrow$ Noncommutative Hodge for $\3\dg$ algebra $\a$. By the results above, suppose $\a$ is a smooth and finite dimensional algebra, then the Hodge conjecture of $\a$ is true.

\begin{defn}
Let $\X$ be a projective smooth variety. An object $\mathsf{T}$ is a called tilting sheaf if the following property holds$\colon$ \\
(1) $\mathsf{T}$ classical generates $\b\X$.\\
(2) $\a:=\mathsf{Hom}(\mathsf{T},\mathsf{T}) $ is of finite global dimension. \\
(3) $\mathsf{Ext}^{k}(\mathsf{T},\mathsf{T})=0$ for $k> 0$.

\par

The reader can refer to Alastair Craw's note, ``Explicit methods for derived categories of sheaves'' \cite{craw} for more discussions.
\end{defn}

Due to Van den Bergh, there are many examples of varieties which admit a tilting bundle.

\begin{eg}(Van den Bergh \cite[theorem A]{2002math......7170V})
 Suppose there is a projective morphism $\f:\X\longrightarrow \Y=\mathsf{Spec}\ \R$ between noetherian schemes. Furthermore, $\mathsf{Rf}_{\ast}(\mathcal{O}_{\X})\cong \mathcal{O}_{\Y}$ and the fibers are at most one dimensional. Then there is a tilting bundle $\mathcal{E}$ of $\X$.
\end{eg}

\begin{cor}\label{Tiltingsheaf}
Suppose $\X$ admits a tilting sheaf, then Hodge conjecture for $\X$ is true.
\end{cor}

\begin{proof}
   Let $\mathsf{T}$ be a tilting sheaf of $\X$. We write $\mathsf{T}$ again after resolution to an injective complex. Define $\a:= \Hom_{\3\dg}(\mathsf{T},\mathsf{T})$, which is quasi-isomorphic (hence derived Morita equivalent) to a smooth and finite dimensional algebra. Thus, the Hodge conjecture for $\X$ is true.
\end{proof}

\end{document}